\documentclass[12pt]{article} % use larger type; default would be 10pt

\usepackage[utf8]{inputenc} % set input encoding (not needed with XeLaTeX)

\usepackage{geometry} % to change the page dimensions
\usepackage{graphicx} % support the \includegraphics command and options

\usepackage{booktabs} % for much better looking tables
\usepackage{array} % for better arrays (eg matrices) in maths
\usepackage{paralist} % very flexible & customisable lists (eg. enumerate/itemize, etc.)
\usepackage{verbatim} % adds environment for commenting out blocks of text & for better verbatim
\usepackage{subfig} % make it possible to include more than one captioned figure/table in a single float
\usepackage{fancyhdr} % This should be set AFTER setting up the page geometry
\usepackage{sectsty}
\allsectionsfont{\sffamily\mdseries\upshape} % (See the fntguide.pdf for font help)
\usepackage[nottoc,notlof,notlot]{tocbibind} % Put the bibliography in the ToC
\usepackage[titles,subfigure]{tocloft} % Alter the style of the Table of Contents

\usepackage{url} % For putting url-s inside the bititems
\usepackage{amsthm, amssymb,amsmath} % I am installing this package in order to be able to write theorems with proofs
\newtheorem{theorem}{Theorem}
\newtheorem{prop}[theorem]{Proposition}

\title{The Subgraphs of Order Six of the Family of Strongly Regular Graphs with Parameters $\lambda=1$ and $\mu=2$}
\author{Reimbay Reimbayev}
\date{} % Activate to display a given date or no date (if empty),
\begin{document}
\maketitle

\begin{abstract}
Strongly regular graphs are highly symmetrical and can be described fully with just a few parameters yet the existence of many of them is still under the question. Due to this uncertainty, it is of immense interest to study their structure, in particular to obtain all the possible subgraphs of lower order. In this paper we study the family of strongly regular graphs with parameters $\lambda =1$ and $\mu =2$ and establish all their subgraphs of order six.
\end{abstract}

\section{Introduction}

The existence of $srg(99,14,1,2)$, famously known as Conway's 99-vertex graph problem \cite{Conway}, is an intriguing one. But this is just one graph from the family of strongly regular graphs for which only few known to exist, e.g. those with valencies $k=2,4$ and, surprisingly 22 \cite{Berlekamp}. For that reason it is of interest to study the structure for the entire family of such graphs rather than just for a particular one. In this paper we have studied all possible subgraphs of order up to six and gave their numbers depending on the order $n$ (or valency $k$, which is interrelated) of such graphs given they do exist. 

As a short reminder, the graph is strongly regular if a pair of its vertices has exactly $\lambda$ common neighbors given they are adjacent, or $\mu$ common neighbors otherwise \cite{Gordon, Brouwer}. Another way of defining strongly regular graphs, perhaps more precise as it cuts away some trivial cases like complete graphs, is by using spectral graph theory by which the finite graph is strongly regular if its spectrum consists of exactly three eigenvalues, one of which is $k$ with multiplicity one \cite{BrouwerMaldeghem}.  

There have been some extensive studies by Makhnev et. al. on the structure of the family of strongly regular graphs with parameters $\lambda =1$ and $\mu =2$ with regard to their automorphism groups \cite{MakhnevMinkova}. Also, Makhnev was able to partially answer to the question of the existence of the graph $srg(99,14,1,2)$ in his earlier work \cite{Makhnev}. Using Wilbrink and Brouwer's lemma \cite{WilbrinkBrouwer}, Lou and Murin were able to establish a forbidden subgraph of order 9 in cade when $k=14$ \cite{LouMurin}. This fact should hold true for any $k$, which needs a strict proof of course. 

In our previous work we have shown that the existance of this graph depends also on number of hexagons it contain \cite{ReiLowerBound}, for which we have set a lower bound. It worth to note that the number of subgraphs of order up to five are all uniquely defined. And only starting with subgraphs of order six the problem of their quantification begins. In this paper we have resolved this problem for six vertex subgraphs setting one of the values, namely $n_3$, as a free variable.

We will divide the paper into subsections in the following matter. First we will bring all the possible six-vertex subgraphs with derivations of their possible numbers. Then the formulas will be summarized at the end of the subsection. One who wishes to skip all the derivations can fast forward to the end of the next subsection.

Also, in order to find some of the values for six-vertex subgraphs it is necessary to know the values for five-vertex subgraphs upfront, which are in our case are all exact and depend only on $n$ and $k$. We will omit the derivations of those values, giving just formulas further, as they can be obtain in a similar but much easier way. Together with five-vertex subgraphs, we will also give values for all four-vertex subgraphs also without derivations.

Finally, we need the next proposition. 

\begin{prop}
Consider $p_i$ for $i=3,4,5$ is the number of triangles, quadrilaterals and pentagons respectively in a strongly regular graph $srg(n,k,1,2)$. Then,
\begin{align*}
p_3=&\frac{1}{6}nk,\\
p_4=&\frac{1}{8}nk(k-2),\\
p_5=&\frac{1}{5}nk(k-2)(k-4).
\end{align*}
\end{prop}
\begin{proof} The proof has given in \cite{ReiLowerBound}. \end{proof}

\section{Derivations of the Formulas}

A strongly regular graphs of sufficiently high order with parameters $\lambda=1$ and $\mu=2$, for simplicity henceforth call it a graph $G$, have exactly 62 possible subgraphs of order six (Figure \ref{mainFigure}). Note that for the graphs of smaller order not all the subgraphs would exist. The subgraphs were obtained through extensive search of all six-vertex graphs that do not break strong-regularity condition for the main graph. Let us denote $N_i$ the graphs of type $i$ from the Figure \ref{mainFigure} and by $n_i$ the number of such graphs in $G$. 

\begin{figure}
	\includegraphics[width=1.0\textwidth]{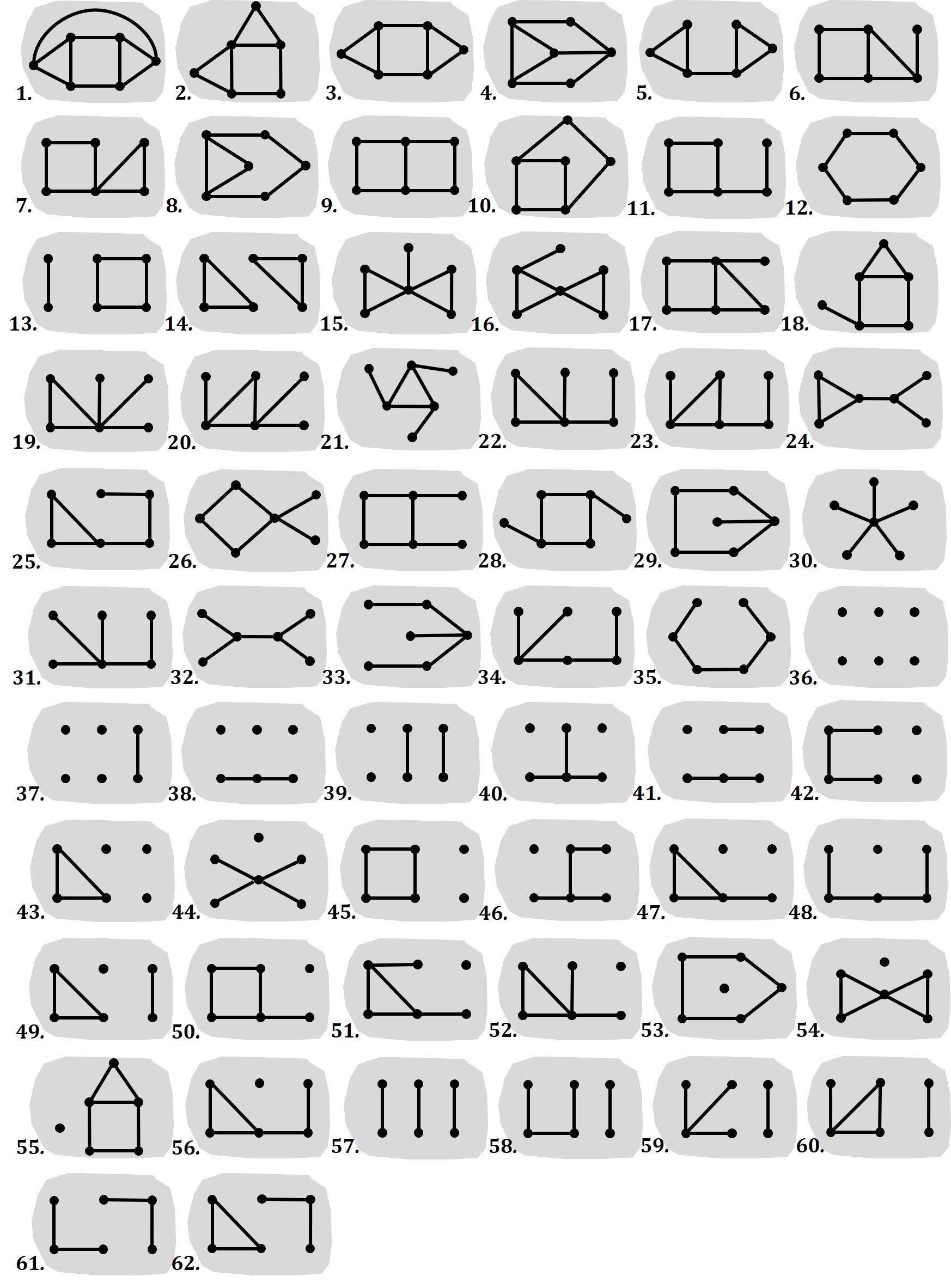}
		\centering
		\caption{All possible subgraphs of order six in $srg(n,k,1,2)$.}
		\label{mainFigure}
\end{figure}

As the exact numbers for six-vertex subgraphs are not known, compared to five-vertex ones, for which it is easy to obtain the exact formulas, we will have to use $n_3$ as a free variable compared to which all other values will be calculated.

For $n_1$, given a quadrilateral, recover triangles on its opposite sides. So we get the relation
\[2p_4=3n_1+n_3.\]
Thus, \[n_1=\frac{1}{12}nk(k-2)-\frac{1}{3}n_3.\]
Notice, that we have left $n_3$ as an unknown parameter through which we will define all other values as noted above.

For $n_2$, given a quadrilateral, choose this time triangles on adjacent sides, \[n_2=4p_4=\frac{1}{2}nk(k-2). \] 

For $n_3$, $n_3=n_3$.

For $n_4$, given a triangle $v_0v_1v_2$, there exist exactly $3(k-2)$ distinct vertices of $G$, each adjacent to exactly one of the vertices of the triangle $v_0v_1v_2$. Choose one, $w_0$, assume it is adjacent to $v_0$. The next two vertices $w_1$ and $w_2$ are predetermined as common neighbors of $w_0$ and $v_1$ and $w_0$ and $v_2$ respectively. Now, $v_0v_1v_2w_0w_1w_2$ can give us a subgraph isomorphic to either $N_1$ or $N_4$. Thus, by this construction we will get: \[3(k-2)p_3=6n_1+n_4.\]
From where, \[n_4=\frac{1}{2}nk(k-2)-6n_1=2n_3.\]

For $n_5$, choose an edge $v_1v_2$. Excluding a triangle based on $v_1v_2$, choose one of the $\frac{k}{2}-1$ incident triangles for each vertex $v_1$ and $v_2$. We obtain: \[|E(G)|(\frac{k}{2}-1)^2=n_5+2n_3+3n_1.\]
And after algebraic operations,
\[n_5=\frac{1}{8}nk(k-2)(k-4)-n_3.\]

For $n_6$, choose a quadrilateral, recover a triangle based on one of its sides. There are exactly $k-2$ vertices adjacent to the third recovered vertex of the triangle. Thus,
\[4(k-2)p_4=n_6+2n_4+6n_1.\]
Plugging the values,
\[n_6=\frac{1}{2}nk(k-2)(k-3)-2n_3.\]

For $n_7$, choose a quadrilateral, at one of its vertices recover one of the $\frac{k}{2}-2$ triangles that are not based on the sides of the quadrilateral. Notice that there cannot be any new edges between the quadrilateral and the triangle. We have \[4p_4(\frac{k}{2}-2)=n_7.\] So, \[n_7=\frac{1}{4}nk(k-2)(k-4).\]

For $n_8$, choose a pentagon. Recover on one of its sides a triangle. Thus, \[5p_5=n_8+n_4.\] And \[n_8=nk(k-2)(k-4)-2n_3.\]

For $n_9$, choose an edge. Due to topology of the graph, there exactly $k-2$ quadrilaterals based on that edge. Choose two. Thus, \[|E(G)|{k-2\choose 2} =n_9+n_4+3n_1.\] After some algebra, \[n_9=\frac{1}{4}nk(k-2)(k-4)-n_3.\]

For $n_{10}$, choose a pentagon in $p_5$ ways. Now we have exactly five pairs of vertices at distance two from each other on that pentagon. To each such pair, there is a unique vertex to recover a quadrilateral on the given pair of vertices. We have \[5p_5=2n_{10}+2n_4.\]
From where, \[n_{10}=\frac{1}{2}nk(k-2)(k-4)-2n_3.\]
Notice that $N_{10}$ will be counted twice by this construction due to two distinct pentagons, while $N_4$ - due to two pairs of vertices of the same pentagon.

For $n_{11}$, let us notice that the graph $N_{11}$ has only one vertex of degree three. We will use that vertex as a starting point for our construction, $n$ ways. Three adjacent to the first vertex but mutually not connected vertices can be chosen in $\frac{1}{6}k(k-2)(k-4)$ ways. Out of these three choose a pair and complete a quadrilateral. By this we have obtain the fifth vertex. To choose the last vertex of degree one (a leaf) we have $k-4$ choices.
\[n\frac{k(k-2)(k-4)}{6}3(k-4)=n_{11}+2n_{10}.\] Or \[n_{11}=\frac{1}{2}nk(k-2)(k-4)(k-6)+4n_3.\]

For $n_{12}$, it has been already found in \cite{ReiLowerBound}. Here we find it in another much easier way in two steps. First, we find the number of paths $P_5$, which we denoted earlier $m_{13}$. Choose a vertex, the middle one in $P_5$; next, two vertices adjacent to it; and finally, two leaves. We have \[n\frac{k(k-2)}{2}(k-3)^2=m_{13}+5p_5.\]
Or \[m_{13}=\frac{1}{2}nk(k-2)(k^2-8k+17).\]

Next construction: choose $P_5$ from $G$. Two leaves of $P_5$ share exactly two neighbors all distinct from other vertices of chosen $P_5$. Choosing one of these neighbors, we have \[2m_{13}=6n_{12}+2n_9+n_8+n_2.\]
Plugging all known values, \[n_{12}=\frac{1}{12}nk(k-2)(2k^2-21k+53)+n_3.\]

For $n_{13}$, choose a quadrilateral. Next, choose an edge of $G$ that is not incident to that quadrilateral. We have: \[p_4(|E(G)|-12-4(k-4))=n_{13}+n_{11}+n_{10}+2n_9+n_7+n_6+2n_4+3n_1.\]
From where \[n_{13}=\frac{1}{32}nk(k-2)(k-4)(k^2-12k+42)-n_3.\]

For $n_{14}$, consider all pairs of triangles that do not share a common vertex. Thus,
\[\frac{1}{2}p_3(p_3-1)-\frac{1}{2}n\cdot\frac{k}{2}(\frac{k}{2}-1)=n_1+n_3+n_5+n_{14}.\]
So, \[n_{14}=\frac{1}{144}nk(k-2)(k-4)(k-12)+\frac{n_3}{3}.\]

For $n_{15}$, choose a vertex. Then choose two adjacent to it triangles and a leaf. We have \[n_{15}=n{k/2 \choose 2}(k-4)=\frac{1}{8}nk(k-2)(k-4).\]

For $n_{16}$, choose a vertex, then two triangles adjacent to it. To one of the four degree-two vertices attach a leaf in $k-4$ ways. \[n_{16}=n{k/2 \choose 2}4(k-4)=\frac{1}{2}nk(k-2)(k-4).\]

For $n_{17}$, choose a quadrilateral, recover on one of its sides a triangle, at one of the vertices of degree three attach a leaf in $k-4$ ways. The graph has been uniquely built.
\[n_{17}=p_4\cdot4\cdot2(k-4)=nk(k-2)(k-4).\]

For $n_{18}$, choose a quadrilateral, recover a triangle on one of its sides. To one of the remaining vertices of the quadrilateral of degree two attach a leaf in $k-4$ ways. We obtain \[p_4\cdot 4 \cdot2(k-4)=n_{18}+2n_4.\] Thus, \[n_{18}=nk(k-2)(k-4)-4n_3.\]

For $n_{19}$, notice that it has only one vertex of degree five. We will start the construction from it. Choose a vertex, a triangle attached to it, and three leaves.
\[n_{19}=n\cdot\frac{k}{2}\cdot \frac{(k-2)(k-4)(k-6)}{6}=\frac{1}{12}nk(k-2)(k-4)(k-6).\] 

For $n_20$, choose a vertex, a triangle attached to it, and two leaves. To one of the two new vertices of the triangle, attach a leaf in $k-4$ ways. Thus,
\[n_{20}=n\cdot\frac{k}{2}\cdot \frac{(k-2)(k-4)}{2}\cdot 2(k-4)=\frac{1}{2}nk(k-2)(k-4)^2.\]

For $n_{21}$, choose a triangle , for every vertex of the triangle choose one of its other $k-2$ neighbors. We get, \[p_3 (k-3)^3=n_{21}+n_6+n_4+2n_1.\] So,\[n_{21}=\frac{1}{6}nk(k-2)(k-3)(k-4)+\frac{2}{3}n_3.\]

For $n_{22}$, once more we look at the vertex of highest degree of the subgraph. To that vertex we attach a triangle in $k/2$ ways, and add two other vertices, to one of which we need to add an extra vertex out of $k-5$ possible choices. Thus,\[n_{22}=n\cdot\frac{k}{2}\cdot \frac{(k-2)(k-4)}{2}\cdot 2(k-5)=\frac{1}{2}nk(k-2)(k-4)(k-5).\]

For $n_{23}$, choose a triangle. To two of its vertices add their neighbors, one to each, there will be $(k-2)(k-3)$ possibilities to do it. Now to one of them add a vertex that would be at distance two from all the vertices of the triangle - in $k-4$ ways.
We have, \[p_3\cdot3(k-2)(k-3)\cdot2(k-4)=n_{23}+2n_8.\]
Thus, \[n_{23}=nk(k-2)(k-4)(k-5)+4n_3.\]

For $n_{24}$ choose an edge. From one side attach a triangle to it in $k/2-1$ ways, to the other side two leaves. Thus we have,
\[|E(G)|\cdot2(\frac{k}{2}-1)\frac{(k-2)(k-4)}{2}=n_{24}+n_{18}+n_4.\]
Or \[n_{24}=\frac{1}{4}nk(k-2)(k-4)(k-6)+2n_3.\]

For $n_{25}$, we start again the construction from the vertex of highest degree of the subgraph. Attach a triangle and a leaf to it. Extend the leaf further by adding to it one of the $k-4$ neighbors that are not adjacent to any other vertices. To the last vertex we will add another its neighbor such that it is still at distance two from the original vertex as well as from the first added leaf - $(k-3)$ ways.
We have \[n\cdot \frac{k}{2}(k-2)(k-4)(k-3)=n_{25}+2n_8.\]
So, \[n_{25}=\frac{1}{2}nk(k-2)(k-4)(k-7)+4n_3.\]

For $n_{26}$, choose a quadrilateral. On one of its vertices attach two leaves. We have, \[n_{26}=p_4 \cdot 4 \cdot \frac{(k-4)(k-6)}{2} =\frac{1}{4}nk(k-2)(k-4)(k-6).\]

For $n_{27}$, choose a quadrilateral, attach two leaves to the neighboring two vertices of it. We have \[p_4\cdot 4(k-4)^2=n_{27}+2n_9.\] So, \[n_{27}=\frac{1}{2}nk(k-2)(k-4)(k-5)+2n_3.\]

For $n_{28}$, choose a quadrilateral. Attach two leaves on its opposite sides. We have \[p_4\cdot 2(k-4)^2=n_{28}+n_{10}.\]. Then \[n_{28}=\frac{1}{4}nk(k-2)(k-4)(k-6)+2n_3.\]

For $n_{29}$, choose a pentagon from $G$. To one of its vertices attach a leaf, which might turn out to be not a leaf but just make sure to not choose a common neighbor with one of its immediate neighbors. We have \[p_5\cdot 5 (k-4)=n_{29}+4n_{10}+n_4.\] And thus, \[n_{29}=nk(k-2)(k-4)(k-6)+6n_3.\]

For $n_{30}$, start from the vertex of highest degree of the subgraph, to which choose five mutually non adjacent its neighbors. Thus, \[n_{30}=n\cdot \frac{1}{5!}k(k-2)(k-4)(k-6)(k-8)=\frac{1}{120}nk(k-2)(k-4)(k-6)(k-8).\]

For $n_{31}$, start from the vertex of highest degree of the subgraph; add four leaves, to one of which add one more vertex in $k-5$ ways. Thus, \[n_{31}=n\cdot \frac{1}{4!}k(k-2)(k-4)(k-6) \cdot 4(k-5)=\frac{1}{6}nk(k-2)(k-4)(k-5)(k-6).\]

For $n_{32}$, choose an edge; to each of its vertices add two non adjacent vertices. We get \[|E(G)|\cdot (\frac{(k-2)(k-4)}{2})^2=n_{32}+n_{27}+n_9.\]
After some calculations, \[n_{32}=\frac{1}{8}nk(k-2)(k-4)(k^2-10k+26)-n_3.\]

For $n_{33}$, choose a vertex of degree three of the subgraph; as a starting point from that vertex we can rebuild the rest of the graph. We have, \[n\cdot \frac{k(k-2)(k-4)}{6}\cdot3(k-4)^2=n_{33}+n_{29}.\]
So \[n_{33}=\frac{1}{2}nk(k-2)(k-4)(k^2-10k+28)-6n_3.\]

For $n_{34}$, similarly to the previous case, we will start the construction from the vertex of degree three and proceed adding vertices. We have,
\[n\cdot \frac{k(k-2)(k-4)}{6}\cdot3(k-4)\cdot (k-3)=n_{34}+2n_{29}+2n_{10}.\]
\[n_{34}=\frac{1}{2}nk(k-2)(k-4)(k^2-11k+34)-8n_3.\]

For $n_{35}$, choose $P_5$ from $G$ in $m_{13}$ ways. Add a vertex to one of its end vertices such that the new vertex will not be connected to any of the two preceding vertices in exactly $k-3$ ways. Thus, 
\[m_{13}\cdot 2(k-3)=2n_{35}+2n_{29}+12n_{12}+2n_8.\]
And \[n_{35}=\frac{1}{2}nk(k-2)(k-4)(k^2-11k+36)-10n_3.\]

Due to not the best choice of numeration of subgrpaphs, further calculations cannot be carried out in linear fashion as the calculation of $n_{36}$ requires the knowledge of the values that come after it. Here we have to skip to $n_{44}$.

For $n_{44}$, reconstruction of the subgraph follows straight forward pattern working out from the vertex of highest degree and further. We have, \[n_{44}=n\cdot \frac{1}{4!}k(k-2)(k-4)(k-6)\cdot(n-k-1-6-4(k-5))=\frac{1}{24}nk(k-2)(k-4)(k-6)(n-5k+13).\]

For $n_{45}$, we choose a quadrilateral of $G$, then choose a pair of vertices that are not adjacent to any vertices of the quadrilateral. This pair can be mutually adjacent, thus giving us \[p_4\cdot {n-8-4(k-4)\choose 2}=n_{45}+n_{13}.\]
So, \[n_{45}=\frac{1}{64}nk(k-2)(k-4)(k-6)(k^2-8k+26)+n_3.\]

For $n_{46}$, start the construction from the vertex of degree three of the subgraph. Choose three mutually non adjacent vertices, one of which continue further. Finally, we need to choose the last vertex our of the ones located at distance two from the first vertex such that it is still non adjacent to two other leaves. Then we have \[n\cdot \frac{k(k-2)(k-4)}{6}\cdot 3(k-4)\cdot (n-k-4-3(k-4))=n_{46}+n_{34}.\]
\[n_{46}=\frac{1}{4}nk(k-2)(k-4)(k^3-14k^2+72k-140)+8n_3.\]

Once more we need to fast forward to $n_{60}$ before we find $n_{47}$.

For $n_{60}$, choose a triangle $K_3$ in $p_3$ ways. Denote $N(K_3)$ a set of vertices in $G$ at distance one from a vertex in $K_3$, and $W(K_3)$ - at distance two from all vertices of $K_3$. In order to built a bigger component of the subgraph we can add one of the vertices from $N(K_3)$ in $3(k-2)$. Now choose a vertex from $W(K_3)$ such that it is not adjacent to the previously chosen vertex. This vertex is adjacent to exactly six vertices from $N(K_3)$. Complete the construction choosing the neighboring vertex that is not one of those six. We have 
\[p_3\cdot 3(k-2)\cdot (|W(K_3)|-(k-4))\cdot (k-6)=2n_{60}+n_{25}.\]
Or, \[n_{60}=\frac{1}{8}nk(k-2)(k-4)(k^2-12k+38)-2n_3.\]

For $n_{47}$, similar to $n_{60}$, we will choose the connected component first and then two non adjacent to that component vertices. We have \[n\cdot k/2\cdot (k-2) {n-k-1 -2(k-2)-(k-4) \choose 2}=n_{47}+n_{60}.\]
Or \[n_{47}=\frac{1}{16}nk(k-2)(k-4)(k-6)(k^2-8k+22)+2n_3.\]

For $n_{48}$, first of all we will show that any $P_3$ from $G$ can be completed to a pentagon in exactly $2(k-4)$ ways.
\begin{prop}
A path $P_3$ in $G$ can be completed to exactly $2(k-4)$ pentagons in $G$.
\end{prop}

\begin{proof}
In order to show that notice that each pentagon contains exactly 5 paths $P_3$. So the average number of pentagons that a given $P_3$ belongs to is $5p_5$ divided to the number of all $P_3$-s in $G$ and that is \[5\cdot \frac{1}{5}nk(k-2)(k-4)/\frac{1}{2}nk(k-2)=2(k-4).\]
That means we just need to prove that this is a maximal number of possible pentagons for a given $P_3$. Assume a path given $v_1v_0v_2$. Denote $w$ another common neighbor to $v_1$ and $v_2$ along with $v_0$, and $z$ a vertex that completes a triangle on edge $v_0v_2$. Now there are exactly $k-3$ neighbors of $v_1$ to which we can continue our path without being connected to $v_0$ or $v_2$. From those $k-3$ vertices there are at most $2(k-3)-2 =2(k-4)$ paths of length two to the vertex $v_2$ that can complete to a pentagon. We subtract two as we have exactly one path that goes through vertex $w$ and one through $z$. The statement is proven.
\end{proof}
Now we proceed with a construction of subgraphs $N_{48}$. Choose a middle vertex of $P_5$ and recover the entire paths, by doing that we will use the previous proposition. The last, possibly isolated, vertex we will choose such that it can only be adjacent to end points of the constructed path in $n-3k+4$ ways. Thus,
\[n\cdot \frac{k(k-2)}{2}\cdot((k-3)^2)-2(k-4)\cdot(n-3k+4)=n_{48}+2n_{35}+6n_{12}.\]
And \[n_{48}=\frac{1}{4}nk(k-2)(k-4)(k^3-14k^2+75k-160)+14n_3.\]

For $n_{62}$, choose a triangle. Similarly, denote $W$ the set of all vertices of $G$ that are at distance two from the three vertices of the triangle. Notice that $G[W]$ is regular of degree $k-6$ and $|W|=n-3-3(k-2)$. Now choose a vertex from $W$ and two its neighbors also from $W$. We have \[p_3\cdot |W| {k-6 \choose 2}=n_{62}+6n_{14}.\]
Or \[n_{62}=\frac{1}{24}nk(k-2)(k-4)(k^2-14k+54)-2n_3.\]

For $n_{49}$, we will make similar construction like for $n_{62}$ starting with a triangle, but this time we choose first an edge from $W$ in $\frac{|W|(k-6)}{2}$ ways, and the last vertex also from $W$ in $|W|-2$ ways. We have \[p_3\cdot \frac{|W|(k-6)}{2}\cdot (|W|-2)=n_{49}+2n_{62}+6n_{14}.\]
So \[n_{49}=\frac{1}{48}nk(k-2)(k-4)(k^3-16k^2+94k-216)+2n_3,\]

For $n_{50}$, we will start the construction from the vertex of degree three of the subgraph. The last vertex, possibly isolated, we choose at distance two from the first vertex such that it is also not adjacent to three other vertices at distance one from the starting vertex. Thus, we have
\[n\cdot\frac{1}{6}k(k-2)(k-4)\cdot 3(n-k-1-2(k-3)-1-(k-4))=n_{50}+2n_{28}.\]
And
\[n_{50}=\frac{1}{4}nk(k-2)(k-4)(k^2-10k+30)-4n_3,\]

For $n_{51}$, choose a triangle, to two of its vertices attach two vertices such that they are not adjacent, $(k-2)(k-3)$ ways. The final, possibly isolated, vertex of the subgraph can be chosen from $W$. We have
\[\frac{nk}{6}\cdot 3(k-2)(k-3)\cdot \frac{1}{2}(k-2)(k-4)=n_{51}+n_{23}+n_{8}.\]
Or \[n_{51}=\frac{1}{4}nk(k-2)(k-4)(k^2-9k+22)-2n_3.\]

Here, for $n_{52}$, we can find its exact value. Start the construction from the vertex of degree four, call it $v_0$. Choose a triangle, in $k/2$ ways, and two leaves, in $\frac{1}{2}(k-2)(k-4)$ ways, attached to $v_0$. From the set of vertices at distance two from $v_0$ we can choose the vertex that is not adjacent to any other already chosen vertices. Thus,
\[n_{52}=n\cdot \frac{k}{2} \cdot \frac{1}{2}(k-2)(k-4) (n-k-1-2(k-2)-2(k-5)-1)=\frac{1}{4}nk(k-2)(k-4)(n-5k+12).\]

For $n_{53}$, choose a pentagon. Notice that all five vertices of the triangles based on the sides of this pentagon are distinct. Then we can choose for the last vertex of the subgraph one of the $n-10$ remaining vertices.
We can get
\[ p_5\cdot (n-10)=n_{53}+2n_{10}+n_{29}.\]
So \[n_{53}=\frac{1}{5}nk(k-2)(k-4)(n-5k+15)-2n_3.\]

For $n_{54}$, start from the vertex of degree four, $v_0$; choose two adjacent to it triangles. From the set of vertices at distance two from $v_0$ choose the last vertex non adjacent to any other previously chosen vertices.
We have \[n_{54}=n\cdot {k/2\choose2}\cdot (n-k-1-2(k-2)-2(k-4)).\]
Or \[n_{54}=\frac{1}{16}nk(k-2)(k-4)(k-6).\]

For $n_{55}$, choose a quadrilateral, recover a triangle at one of its sides. The last vertex choose such that it is not adjacent to any vertices of the quadrilateral. Thus we have
\[p_4\cdot 4(n-8-4(k-4))=n_{55}+n_6.\]
And \[n_{55}=\frac{1}{4}nk(k-2)(k-4)(k-6)+2n_3.\]

For $n_{56}$, start the construction from the vertex of degree three, call it $v_0$. Choose a triangle and a leaf attached to $v_0$. Continue the leaf further attaching another vertex to it. The sixth vertex we choose from the set of vertices at distance two from $v_0$ with extra requirements of not being adjacent to any of the chosen vertices except the very last one. Then we have \[n\cdot \frac{k}{2}\cdot (k-2) (k-4)(n-k-2-2(k-2)-(k-4))=n_{56}+n_{25}.\]
Or \[n_{56}=\frac{1}{2}nk(k-2)(k-4)(n-5k+14)-4n_3.\]

Once more, we need the value of $n_{58}$ before looking for $n_{57}$. For that notice that we have a relation \[m_6\cdot2(k-3)=2n_{58}+2n_{35}+n_{25}.\] This is due to the fact that in five-vertex subgraph $M_6$ we can prolong its component $P_3$ path to $P_4$ by adding to either of its sides a vertex in $k-3$ ways. Thus,
\[n_{58}=\frac{1}{8}nk(k-2)(k-4)(k^3-15k^2+86k-190)+8n_3.\]   

Now for $n_{57}$, start the construction from choosing an isolated edge, in $|E(G)|$ ways. Next choose a pair of vertices from those what are not incident to the chosen one. There will be exactly $|E(G)|-3-7(k-2)-2(k-2)(k-4)$ edges to choose the pair from. Notice that by this construction we might end up with a five-vertex subgraph. Thus
\[|E(G)|{|E(G)|-3-7(k-2)-2(k-2)(k-4) \choose 2} =3m_{14}+m_6+n_{60}+2n_{13}+n_{58}+3n_{57}.\]
And \[n_{57}=nk(k-2)(k-4)(k^4-18k^3+140k^2 - 564k + 996)/192-\frac{4}{3}n_3.\]

For $n_{59}$, start the construction from the vertex of degree three, $v_0$. Choose three other vertices from its neighborhood. Lastly, choose an edge from the set of vertices at distance two from $v_0$ such that it is not a base for any triangle on previously chosen vertices. We get \[n\cdot\frac{1}{6}k(k-2)(k-4)\cdot(|E(G)|-\frac{3}{2}k-k(k-2)-3(\frac{k}{2}-1))=n_{59}+n_{34}+n_{29}+2n_{28}+2n_{10}.\]
Or \[n_{59}=\frac{1}{24}nk(k-2)(k-4)(k-6)(k^2-10k+34)+2n_3.\]

For $n_{61}$, consider five-vertex subgraph $M_3$. Notice that two isolated vertices should have exactly two common neighbors due to the topology of the graph. Then
\[2m_3=2n_{61}+2n_{32}+n_{34}+n_{20}+n_{26}.\]
From which \[n_{61}=\frac{1}{16}nk(k-2)(k-4)(k^3-16k^2 +96k -220)+5n_3.\]

The rest of the values $n_{36}-n_{43}$ can be found from the equations directly. But we will stick to constructions wherever it is possible so we can use the equations as the way of check. It will also decrease the calculations tremendously. Again we have to go from top to bottom.

For $n_{43}$, choose a triangle. From the set $W$ of vertices at distance two from all the vertices of the triangle, choose three vertices. We have
\[p_3{W\choose 3}=n_{43}+n_{62}+n_{49}+2n_{14}.\]
And \[n_{43}=\frac{1}{288}nk(k-2)(k-4)(k^4-18k^3+130k^2-460k+720)-\frac{2}{3}n_3.\]

For $n_{42}$, consider five-vertex subgrpaph $M_3$. The component $P_3$ in it we can continue in $k-3$ was in either direction. By that we have
\[m_3\cdot 2(k-3)=2n_{42}+2n_{48}+n_{34}.\]
Thus \[n_{42}=\frac{1}{16}nk(k-2)(k-4)(k^4-17k^3+120k^2-430k+684)-10n_3.\]

For $n_{41}$, consider again $M_3$. Two isolated vertices of the subgraph has exactly two common neighbors so there exactly $2(k-2)$ vertices that are neighbors to only one of them. Thus, we obtain a relationship
\[m_3\cdot 2(k-2)=2n_{41}+2n_{48}+n_{46}+2n_{51}+n_{50}.\]
Or
\[n_{41}=\frac{1}{16}nk(k-2)(k-4)(k^4-18k^3+136k^2-524k+892)-14n_3.\]

For $n_{40}$, we start the construction from the vertex $v_0$ of degree three. Then choose three mutually non-adjacent its neighbors. Lastly, choose the pair of vertices from the set of vertices at distance two from $v_0$ such that none is adjacent to any previously chosen vertices. We have
\[n\cdot \frac{k(k-2)(k-4)}{6}\cdot {n-k-1-3-3(k-4) \choose 2}=n_{40}+n_{54}.\]
Or \[n_{40}=\frac{1}{48}nk(k-2)(k-4)(k^4-18k^3+130k^2-460k+696)-2n_3.\]

For $n_{39}$, consider five-vertex subgraph $M_4$. There are exactly $n-k-5$ vertices of the graph $G$ excluding the ones that are already in $M_4$ that are not adjacent to the only isolated vertex of the subgraph. Thus we have
\[m_4\cdot (n-k-5)=n_{54}+n_{56}+3n_{49}+n_{48}+2n_{41}+2n_{39}.\]
Or \[n_{39}=\frac{1}{128}nk(k-2)(k-4)(k^5 - 20k^4 + 176k^3 - 884k^2 + 2588k - 3624)+6n_3.\]

For $n_{38}$, we start the construction from the vertex $v_0$ of degree two, then choose its two neighbors. We have obtained $P_3$. Lastly, we choose three vertices from the set of all vertices of $G$ that are not adjacent to any of the vertices of $P_3$. We have
\[ n \cdot \frac{k(k-2)}{2}\cdot {n-3k+4\choose 3}=n_{38}+n_{41}+2n_{61}+n_{62}.\]
Or
\[n_{38}=\frac{1}{96}nk(k-2)(k-4)(k^5-20k^4+172k^3-828k^2+2300k-3048)+6n_3.\]

For $n_{37}$, consider five-vertex subgraph $M_1$. To one of its five isolated vertices add its neighboring vertex. We have
\[m_1\cdot 5k=2n_{37}+2n_{38}+3n_{40}+4n_{44}+5n_{30}.\]
And \[n_{37}=\frac{1}{768}nk(k-2)(k-4) (k^6 - 22k^5 + 212k^4 - 1208k^3 + 4484k^2 - 10456k + 12288)-3n_3.\]

And finally, for $n_{36}$ we use the equation
\[m_1(n-5)=n_{30}+6n_{36}+2n_{37}+n_{38}+n_{40}+n_{44}.\]
So \[n_{36}=\frac{1}{23040}nk(k-2)(k-4)(k^7-24k^6+248k^5-1520k^4+6436k^3-19520k^2+38896k-40704)+\frac{1}{3}n_3.\]

Thus we have obtained all the formulas for the number of subgraphs of order six in $G$. Below is the summary of all our calculations of $n_i$-s.

\begingroup
\allowdisplaybreaks
 \begin{flalign*}
 n_1=&\frac{1}{12}nk(k-2)-\frac{n_3}{3},&\\
 n_2=&\frac{1}{2}nk(k-2),\\
 n_3=&n_3,\\
 n_4=&2n_3,\\
 n_5=&\frac{1}{8}nk(k-2)(k-4)-n_3,\\
 n_6=&\frac{1}{2}nk(k-2)(k-3)-2n_3,\\
 n_7=&\frac{1}{4}nk(k-2)(k-4),\\
 n_8=&nk(k-2)(k-4)-2n_3,\\
 n_9=&\frac{1}{4}nk(k-2)(k-4)-n_3,\\
 n_{10}=&\frac{1}{2}nk(k-2)(k-4)-2n_3,\\
 n_{11}=&\frac{1}{2}nk(k-2)(k-4)(k-6)+4n_3,\\
n_{12}=&\frac{1}{12}nk(k-2)(2k^2-21k+53)+n_3,\\
n_{13}=&\frac{1}{32}nk(k-2)(k-4)(k^2-12k+42)-n_3,\\
n_{14}=&\frac{1}{144}nk(k-2)(k-4)(k-12)+\frac{n_3}{3},\\
n_{15}=&\frac{1}{8}nk(k-2)(k-4),\\
n_{16}=&\frac{1}{2}nk(k-2)(k-4),\\
n_{17}=&nk(k-2)(k-4),\\
n_{18}=&nk(k-2)(k-4)-4n_3,\\
n_{19}=&\frac{1}{12}nk(k-2)(k-4)(k-6),\\
n_{20}=&\frac{1}{2}nk(k-2)(k-4)^2,\\
n_{21}=&\frac{1}{6}nk(k-2)(k-3)(k-4)+\frac{2n_3}{3},\\
n_{22}=&\frac{1}{2}nk(k-2)(k-4)(k-5),\\
n_{23}=&nk(k-2)(k-4)(k-5)+4n_3,\\
n_{24}=&\frac{1}{4}nk(k-2)(k-4)(k-6)+2n_3,\\
n_{25}=&\frac{1}{2}nk(k-2)(k-4)(k-7)+4n_3,\\
n_{26}=&\frac{1}{4}nk(k-2)(k-4)(k-6),\\
n_{27}=&\frac{1}{2}nk(k-2)(k-4)(k-5)+2n_3,\\
n_{28}=&\frac{1}{4}nk(k-2)(k-4)(k-6)+2n_3,\\
n_{29}=&nk(k-2)(k-4)(k-6)+6n_3,\\
n_{30}=&\frac{1}{120}nk(k-2)(k-4)(k-6)(k-8),\\
n_{31}=&\frac{1}{6}nk(k-2)(k-4)(k-5)(k-6),\\
n_{32}=&\frac{1}{8}nk(k-2)(k-4)(k^2-10k+26)-n_3,\\
n_{33}=&\frac{1}{2}nk(k-2)(k-4)(k^2-10k+28)-6n_3,\\
n_{34}=&\frac{1}{2}nk(k-2)(k-4)(k^2-11k+34)-8n_3,\\
n_{35}=&\frac{1}{2}nk(k-2)(k-4)(k^2-11k+36)-10n_3,\\
n_{36}=&\frac{1}{23040}nk(k-2)(k-4)(k^7-24k^6+248k^5-1520k^4+6436k^3-19520k^2\\
&+38896k-40704)+\frac{1}{3}n_3,\\
n_{37}=&\frac{1}{768}nk(k-2)(k-4) (k^6 - 22k^5 + 212k^4 - 1208k^3 + 4484k^2 - 10456k + 12288)-3n_3,\\
n_{38}=&\frac{1}{96}nk(k-2)(k-4)(k^5-20k^4+172k^3-828k^2+2300k-3048)+6n_3,\\
n_{39}=&\frac{1}{128}nk(k-2)(k-4)(k^5 - 20k^4 + 176k^3 - 884k^2 + 2588k - 3624)+6n_3,\\
n_{40}=&\frac{1}{48}nk(k-2)(k-4)(k^4-18k^3+130k^2-460k+696)-2n_3,\\
n_{41}=&\frac{1}{16}nk(k-2)(k-4)(k^4-18k^3+136k^2-524k+892)-14n_3,\\
n_{42}=&\frac{1}{16}nk(k-2)(k-4)(k^4-17k^3+120k^2-430k+684)-10n_3,\\
n_{43}=&\frac{1}{288}nk(k-2)(k-4)(k^4-18k^3+130k^2-460k+720)-\frac{2}{3}n_3,\\
n_{44}=&\frac{1}{24}nk(k-2)(k-4)(k-6)(n-5k+13),\\
n_{45}=&\frac{1}{64}nk(k-2)(k-4)(k-6)(k^2-8k+26)+n_3,\\
n_{46}=&\frac{1}{4}nk(k-2)(k-4)(k^3-14k^2+72k-140)+8n_3,\\
n_{47}=&\frac{1}{16}nk(k-2)(k-4)(k-6)(k^2-8k+22)+2n_3,\\
n_{48}=&\frac{1}{4}nk(k-2)(k-4)(k^3-14k^2+75k-160)+14n_3,\\
n_{49}=&\frac{1}{48}nk(k-2)(k-4)(k^3-16k^2+94k-216)+2n_3,\\
n_{50}=&\frac{1}{4}nk(k-2)(k-4)(k^2-10k+30)-4n_3,\\
n_{51}=&\frac{1}{4}nk(k-2)(k-4)(k^2-9k+22)-2n_3,\\
n_{52}=&\frac{1}{4}nk(k-2)(k-4)(n-5k+12),\\
n_{53}=&\frac{1}{5}nk(k-2)(k-4)(n-5k+15)-2n_3,\\
n_{54}=&\frac{1}{16}nk(k-2)(k-4)(k-6),\\
n_{55}=&\frac{1}{4}nk(k-2)(k-4)(k-6)+2n_3,\\
n_{56}=&\frac{1}{4}nk(k-2)(k-4)(k^2-10k+30)-4n_3,\\
n_{57}=&nk(k-2)(k-4)(k^4-18k^3+140k^2 - 564k + 996)/192-\frac{4}{3}n_3,\\
n_{58}=&\frac{1}{8}nk(k-2)(k-4)(k^3-15k^2+86k-190)+8n_3,\\
n_{59}=&\frac{1}{24}nk(k-2)(k-4)(k-6)(k^2-10k+34)+2n_3,\\
n_{60}=&\frac{1}{8}nk(k-2)(k-4)(k^2-12k+38)-2n_3,\\
n_{61}=&\frac{1}{16}nk(k-2)(k-4)(k^3-16k^2 +96k -220)+5n_3,\\
n_{62}=&\frac{1}{24}nk(k-2)(k-4)(k^2-14k+54)-2n_3.\\
\end{flalign*}
\endgroup

First of all notice that the values for $n_3$ are not given. The arguments of symmetry tell us that it should be equal zero, but then it would follow immediately that $srg(99,14,1,2)$, an infamous strongly regular of Conway, doesn't exist \cite{Makhnev}. All other values for other six-vertex subgraphs are given in terms of $n,k$ and $n_3$. Of course, it is possible using the relation between $n$ and $k$ to get rid of $n$ altogether but then the formulas become very cumbersome.

\section{Four- and Five-Vertex Subgraphs}

There are only nine subgraphs of order four in $G$. They are given in Figure \ref{allFoursFigure}. Denote them $L_i$ and their quantities $l_i (i=\overline{1,9})$. It is a trivial exercise to find the number of each subgraph of order four. They are as follows:

\begin{figure}
	\includegraphics[width=0.7\textwidth]{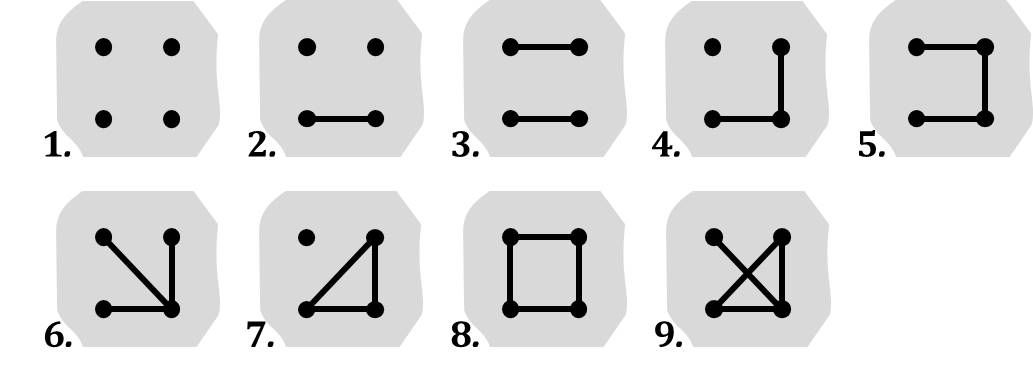}
		\centering
		\caption{All possible subgraphs of order four in $srg(n,k,1,2)$.}
		\label{allFoursFigure}
\end{figure}

\begin{align*}
l_1=&\frac{1}{192}nk(k-2)(k-4)(k^3-6k^2+10k-12),\\
l_2=&\frac{1}{16}nk(k-2)(k-4)(k^2-4k+6),\\
l_3=&\frac{1}{16}nk(k-2)(k^2-6k+10),\\
l_4=&\frac{1}{2}nk(k-2)(n-3k+4),\\
l_5=&\frac{1}{2}nk(k-2)(k-3),\\
l_6=&\frac{1}{6}nk(k-2)(k-4),\\
l_7=&\frac{1}{12}nk(k-2)(k-4),\\
l_8=&\frac{1}{8}nk(k-2),\\
l_9=&\frac{1}{2}nk(k-2).
\end{align*}
Based on these values, we can find the values for the five-vertex subgraphs. For that, we can use the following equations bounding the values for four-vertex subgraphs with five-vertex ones. To obtain these equations we need to consider all the possible graphs that can be obtained by adding to a given four-vertex subgraph exactly one vertex out of $n-4$ remaining in $G$. On the other hand, we can obtain the same quantities by subtracting a vertex from a given five-vertex subgraph. It is a tedious but straight-forward exercise.
\begin{align*}
l_1(n-4)=&5m_1+2m_2+m_3+m_5+m_9,\\
l_2(n-4)=&3m_2+2m_3+4m_4+m_6+2m_7+3m_8+m_{11}+m_{12}+m_{17},\\
l_3(n-4)=&m_4+2m_6+m_{13}+3m_{14}+m_{19}+m_{21},\\
l_4(n-4)=&2m_3+3m_5+2m_6+2m_7+4m_{10}+m_{11}+2m_{12}+2m_{13}+m_{15}+2m_{16},\\
l_5(n-4)=&m_7+2m_{11}+2m_{13}+2m_{15}+m_{16}+5m_{18}+2m_{20}+2m_{21},\\
l_6(n-4)=&m_5+4m_9+m_{11}+m_{15}+2m_{17},\\
l_7(n-4)=&2m_8+m_{12}+2m_{14}+m_{21},\\
l_8(n-4)=&m_{10}+m_{15}+m_{20},\\
l_9(n-4)=&m_{12}+2m_{16}+2m_{17}+4m_{19}+2m_{20}+m_{21}.
\end{align*}

By $M_i$ here we denote the subgraph of type $i$ from the Figure \ref{allFivesFigure} and $m_i$ - the number of such subgraphs in $G$. Notice, we can also use these equations to check the values for $m_i$-s. They are given below. Figure \ref{allFivesFigure} depicts their configurations.

\begingroup
\allowdisplaybreaks
 \begin{flalign*}
 m_1=&\frac{1}{960}nk(k-2)(k-4)(n-4k+6)(k^3-6k^2+14k-36),&\\
 m_2=&\frac{1}{96}nk(k-2)(k-4)^2 (k^3-8k^2+26k-48),\\
 m_3=&\frac{1}{16}nk(k-2)(k-4)(k^3-10k^2+38k-60),\\
 m_4=&\frac{1}{32}nk(k-2)(k-4)(k^3-10k^2+40k-68),\\
 m_5=&\frac{1}{6}nk(k-2)(k-4)(n-4k+8),\\
 m_6=&\frac{1}{8}nk(k-2)(k-4) (k^2-8k+20),\\
 m_7=&\frac{1}{4}nk(k-2)(k-4)(k^2-7k+16),\\
 m_8=&\frac{1}{24}nk(k-2)(k-4)(n-4k+8),\\
 m_9=&\frac{1}{24}nk(k-2)(k-4)(k-6),\\
 m_{10}=&\frac{1}{8}nk(k-2)(n-4k+8),\\
 m_{11}=&\frac{1}{2}nk(k-2)(k-4)^2,\\
 m_{12}=&\frac{1}{4}nk(k-2)(k-4)^2,\\
 m_{13}=&\frac{1}{2}nk(k-2)(k^2-8k+17),\\
 m_{14}=&\frac{1}{24}nk(k-2)(k-4)(k-6),\\
 m_{15}=&\frac{1}{2}nk(k-2)(k-4),\\
 m_{16}=&\frac{1}{2}nk(k-2)(k-3),\\
 m_{17}=&\frac{1}{4}nk(k-2)(k-4),\\
 m_{18}=&\frac{1}{5}nk(k-2)(k-4),\\
 m_{19}=&\frac{1}{8}nk(k-2),\\
 m_{20}=&\frac{1}{2}nk(k-2),\\
 m_{21}=&\frac{1}{2}nk(k-2)(k-4).\\
\end{flalign*}
\endgroup

\begin{figure}
	\includegraphics[width=1.0\textwidth]{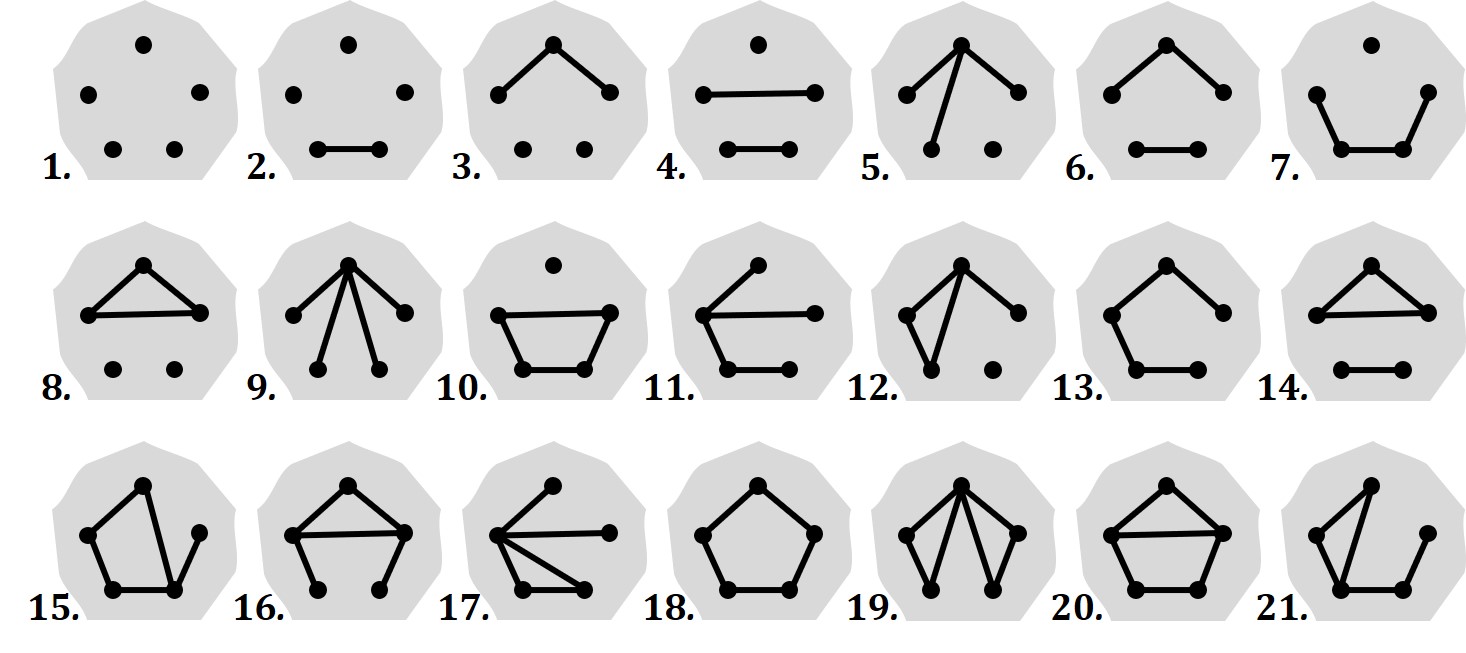}
		\centering
		\caption{All possible subgraphs of order five in $srg(n,k,1,2)$.}
		\label{allFivesFigure}
\end{figure}

Similarly, we can derive the equations bounding the values for $n_i$-s with the values for  $m_i$-s. In fact we can do it for subgraphs of any orders given that we studied all the possible configurations of them.
\begingroup
\allowdisplaybreaks
 \begin{flalign*}
m_1(n-5)=&n_{30}+6n_{36}+2n_{37}+n_{38}+n_{40}+n_{44},&\\
m_2(n-5)=&n_{19}+n_{31}+4n_{37}+2n_{38}+4n_{39}+n_{41}+2n_{42}+3n_{43}+n_{46}+n_{47}+n_{52}+n_{59},\\
m_3(n-5)=&n_{20}+n_{26}+2n_{32}+n_{34}+3n_{38}+3n_{40}+ 2n_{41}+ 2n_{42}+ 4n_{45} +n_{46}+2n_{47}+2n_{48}\\
&+n_{50}+2n_{51}+2n_{61},\\
m_4(n-5)=&n_{15}+n_{22}+n_{33}+2n_{39}+2n_{41}+n_{48}+3n_{49}+n_{54}+n_{56}+6n_{57}+2n_{58}+n_{60},\\
m_5(n-5)=&n_{20}+2n_{28}+n_{31}+n_{34}+2n_{40}+4n_{44}+n_{46}+n_{50}+2n_{52}+2n_{59},\\
m_6(n-5)=&n_{7}+n_{11}+4n_{13}+n_{16}+n_{23}+n_{24}+n_{25}+n_{34}+ 2n_{35}+ n_{41}+2n_{58}+3n_{59}\\
&+2n_{60}+4n_{61}+3n_{62},\\
m_7(n-5)=&n_{17}+3n_{21}+2n_{27}+n_{29}+2n_{33}+2n_{35}+2n_{42} +2n_{46}+2n_{48}+2n_{50}+n_{51}+5n_{53}\\
&+2n_{55}+2n_{56}+2n_{58},\\
m_8(n-5)=&n_{24}+3n_{43}+n_{47}+2n_{49}+n_{56}+n_{62},\\
m_9(n-5)=&2n_{19}+n_{26}+5n_{30}+n_{31}+n_{44},\\
m_{10}(n-5)=&n_{6}+n_{11}+2n_{13}+2n_{45}+n_{50}+n_{55},\\
m_{11}(n-5)=&2n_{10}+n_{11}+n_{17}+n_{18}+n_{20}+2n_{22}+ 2n_{24}+ 2n_{26}+ 2n_{27}+ 2n_{29}+3n_{31}+4n_{32}\\
&+2n_{33}+n_{34}+n_{46},\\
m_{12}(n-5)=&n_{16}+n_{18}+n_{22}+n_{23}+n_{25}+2n_{47}+ 2n_{51}+ 2n_{52}+ 4n_{54}+ 2n_{55}+ n_{56}+ 2n_{60},\\
m_{13}(n-5)=&n_{2}+2n_{6}+2n_{8}+2n_{9}+2n_{11}+6n_{12}+n_{18}+n_{23}+ 2n_{25}+ 2n_{28}+ 2n_{29}+ n_{33}\\
&+2n_{34}+2n_{35}+n_{48},\\
m_{14}(n-5)=&2n_{5}+6n_{14}+n_{25}+n_{49}+n_{60}+2n_{62},\\
m_{15}(n-5)=&2n_{4}+2n_{7}+4n_{9}+2n_{10}+n_{11}+n_{17}+n_{18}+2n_{26}+ 2n_{27}+ 2n_{28}+ n_{50},\\
m_{16}(n-5)=&2n_{2}+n_{4}+2n_{6}+n_{8}+2n_{16} +n_{17} +2n_{20} +3n_{21}+n_{23} +n_{51},\\
m_{17}(n-5)=&n_{7}+4n_{15}+n_{17}+3n_{19}+n_{20}+n_{22}+n_{52},\\
m_{18}(n-5)=&n_{4}+n_{8}+2n_{10}+n_{29}+n_{53},\\
m_{19}(n-5)=&n_{2}+n_{15}+n_{16}+n_{54},\\
m_{20}(n-5)=&6n_{1}+2n_{2}+2n_{3}+2n_{4}+n_{6}+n_{17}+n_{18}+n_{55},\\
m_{21}(n-5)=&4n_{3}+4n_{5}+2n_{7}+2n_{8}+n_{16}+n_{18}+n_{22}+n_{23}+2n_{24}+n_{25}+n_{56}.\\
\end{flalign*}
\endgroup

\section{Conclusion}

In this paper we have studied the structure of a class of strongly regular graphs with parameters $\lambda=1$ and $\mu=2$. In particular we have found the subgraphs of order up to six: all their configurations and their number if the graph with given parameters do exist. All the subgraphs of order up to five have exact values that depend only on parameters $n$ and $k$, which are themselves interdependent. In fact, we can say that they depend only on valency of $G$. While starting from subgraphs of order six it is required one additional parameter $n_3$ in order to define the number of each subgraph in $G$. Well, at least we were not able to avoid it. All the arguments of symmetry tell that $n_3$ must be equal zero. But that would immediately mean, as shown by Makhnev \cite{Makhnev}, that an $srg(99,14,1,2)$ doesn't exist, while the rest of the graphs from the family must be of more coerce structure, namely consist of Paley-9 as building blocks.

Some preliminary research done on subgraphs of order seven tells us that their numbers depend on two parameters one of which can be chosen the same $n_3$. This is clearly a direction to work on although no conclusions might be guaranteed.

%\subsection{A subsection}

\end{document}